\documentclass[graybox]{svmult}
\usepackage{type1cm}        
\usepackage{makeidx}         
\usepackage{graphicx}        
\usepackage{multicol}        
\usepackage[bottom]{footmisc}
\usepackage{mathtools}
\usepackage{xcolor}
\usepackage{newtxtext}       %
\usepackage[varvw]{newtxmath}       
\makeindex           
\usepackage{amsthm}
\usepackage[pdftex,
            colorlinks=true,
            linkcolor=black,
            citecolor=green,
            filecolor=black,
            urlcolor=green]{hyperref}
\begin{document}

\title*{Polynomially Solvable Cases of the Quadratic Binary Optimization Problem with a Fixed Cardinality Constraint}
\author{Thi Thanh Tu Le\orcidID{0009-0001-8161-4535}
and\\ Thorsten Koch\orcidID{0000-0002-1967-0077}} 
\titlerunning{Polynomial-time solvability of the FCBQP} 
\institute{Thi Thanh Tu Le \at Software and Algorithms for Discrete Optimization, Technische Universität Berlin, Germany; and  Applied Optimization Department, Zuse Institute Berlin, Germany.\\ \email{thi.le@zib.de}
\and Thorsten Koch \at Chair of Software and Algorithms for Discrete Optimization, Technische Universität Berlin, Germany; and Applied Optimization Department, Zuse Institute Berlin, Germany.\\ \email{koch@zib.de}}
\maketitle
\abstract{Quadratic Unconstrained Binary Optimization (QUBO) problems and their variant with a fixed cardinality constraint are both well-known NP-hard problems. While certain structural properties of the cost matrix can make some QUBO instances solvable in polynomial time, this advantage typically disappears once a cardinality constraint is added. In this work, we investigate specific structural conditions on the cost matrix under which the constrained version--quadratic binary optimization with a fixed cardinality--remains tractable.}



\section{Introduction}
The \emph{Quadratic Unconstrained Binary Optimization} (QUBO) problem can be stated as
\begin{equation} \label{eq:QUBO}
    \min_{x \in \{0,1\}^n}x^\top Qx + c^\top x,
\end{equation} 
where $Q = (q_{ij})_{1\leq i,j\leq n} \in \mathbb{R}^{n \times n}$ is a cost matrix and $c \in \mathbb{R}^n$. 
QUBOs are related to well-known combinatorial optimization problems (COPs), including the Ising model, Max-Cut, set partitioning, set packing, and quadratic assignment problems. 
The QUBO problem is NP-hard in general, as it is equivalent to the NP-hard Ising model \cite{barahona1982computational}. 

Several classic COPs that can be formulated as QUBOs, such as the traveling salesperson problem, Sudoku, and the quadratic assignment problem, incorporate cardinality constraints implicitly. In other cases, a fixed cardinality constraint, of the type $\sum_{i=1}^nx_i=k$, is explicitly added to QUBO models for certain applications -- for example,  the maximal dispersion problem or the fixed cardinality stable set. 
In the latter case, although the number of feasible solutions reduces from $2^n$ to $\binom{n}{k}$, the problem remains NP-hard, as it generalizes the decision problem of finding the maximum clique. 

Certain structural conditions on the cost matrix can lead to polynomial-time solvability of QUBOs, though this tractability does not extend to its fixed-cardinality variant. For example, in the submodular case, where the cost matrix $Q$ has non-positive off-diagonal elements, Problem \eqref{eq:QUBO} can be solved in polynomial time by reducing it to a minimum $s$-$t$ cut (see \cite[Proposition 7]{caprara2008}). However, the constrained version is equivalent to the NP-hard minimum $s$-$t$ cut problem with fixed cardinality shores (see \cite[Theorem 4.3]{chen2016size}).

In this paper, we investigate specific structural conditions on the cost matrix under which the quadratic binary program with a fixed cardinality constraint can be solved in polynomial time. The problem formulation and easily verifiable tractable cases are mentioned in Section 2. Section 3 continues with linear conditions on the cost matrix, while Section 4 explores rank-based properties of the cost matrix that lead to polynomial-time solvability. Finally, in Section 5, we summarize our contributions and perspectives.
\section{Formulation and simple polynomial-time cases} \label{sec:chapter_2}
The \emph{Fixed Cardinality Constrained Binary Quadratic Program} (FCBQP) is defined as
\begin{align*} 
  \min \quad \biggl\{ x^\top Qx + c^\top x \;\bigg|\; \sum_{i=1}^n x_i = k,\ x \in \{0,1\}^n \biggr\},
\end{align*}
where $Q = (q_{ij})_{1 \leq i,j \leq n} \in \mathbb{R}^{n \times n}$ is the cost matrix, $c \in \mathbb{R}^n$ is an $n$-dimensional vector, $k$ is an integer, and $ 1\leq k \leq n.$ We refer to this problem as FCBQP($Q,c, k$).

One can easily verify that certain simple structures of the cost matrix allow FCBQPs to be solved in polynomial time. Some of these structures are listed below:
 \begin{itemize}
     \item  If all entries of $Q$ are equal, the FCBQP($Q,c, k$) can be solved in $\mathcal{O}(n)$ time. Note that selecting $k$ smallest elements from an unsorted array can be done in $\mathcal{O}(n)$ time using the Median of Medians algorithm.
     \item If $Q$ is diagonal, i.e., $q_{ij} =0$ for all $i \neq j$  (including the case $Q$ is a zero matrix), then the FCBQP($Q,c, k$) (reducing to a linear program) can be solved in $\mathcal{O}(n)$ time. 
 \end{itemize}
Structural properties of the cost matrix that result in the tractability of QUBOs are discussed in Section 3.3 of \cite{punnen2022quadratic}. Inspired by their research ideas, we also consider these restrictions, but in the context of FCBQPs. 
\section{Linear restrictions on the cost matrices}
 A matrix $Q$ is said to be a \textit{sum matrix} (as defined in \cite[Section 3.3.1]{punnen2022quadratic}) if there exist vectors $a,b \in \mathbb{R}^n$ such that $q_{ij} = a_i + b_j$ for all $1 \leq i, j \leq n$.
 
It is shown in \cite[Theorem 3.12]{punnen2022quadratic} that QUBOs can be solved in $\mathcal{O}(n^2)$ time when the cost matrix $Q$ is a sum matrix. By adapting the proof technique and incorporating the fixed cardinality constraint into the objective function, we show that the complexity reduces to $\mathcal{O}(n)$ when this constraint is imposed.
 \begin{theorem}
     If $Q$ is a sum matrix, the FCBQP($Q,c, k$) can be solved in $\mathcal{O}(n)$ time.
     \label{theorem 1}
 \end{theorem}
 \begin{proof}
     With the assumption that $q_{ij} = a_i + b_j$ for all $1 \leq i, j \leq n$, we obtain:
     \begin{align*}
        x^\top Q x + c^\top x = \sum_{i=1}^n\sum_{j=1}^nq_{ij}x_ix_j + \sum_{i=1}^nc_ix_i 
        = \left( \sum_{i=1}^n (a_i + b_i) x_i \right) \left( \sum_{i=1}^n x_i \right) + \sum_{i=1}^nc_ix_i.
     \end{align*}
     Using the cardinality constraint $\sum_{i=1}^nx_i =k$, the problem becomes:
     \begin{equation*}
    \min \left\{ \sum_{i=1}^n \left( k(a_i + b_i) + c_i \right) x_i \ \middle| \ \sum_{i=1}^n x_i = k,\ x \in \{0,1\}^n \right\}.
\end{equation*}
 The optimal solution is constructed as follows: we set $x_i = 1$ for the $k$ variables corresponding to the smallest values of $(k(a_i+b_i) +c_i) $ in $n$ variables, and set $x_i = 0$ for the remaining ones. This can be done in $\mathcal{O}(n)$ time.  
 \end{proof}
 A matrix $Q$ is said to be a \emph{weak sum matrix} if there exist vectors $a,b  \in \mathbb{R}^n$ such that $q_{ij}=a_i+b_j$ for all $1 \leq i,j \leq  n$ and $i \neq j.$  This case can be reduced to the sum matrix case \cite[Corollary 3.3]{punnen2022quadratic}. Combining with Theorem \ref{theorem 1}, we obtain the following result.
 \begin{corollary}
      The FCBQP($Q,c, k$) can be solved in $\mathcal{O}(n)$ time if $Q$ is a weak sum matrix. 
 \end{corollary}
The \emph{Quadratic Assignment Problem} (QAP) is a well-studied COP. Let $D=(d_{ij})_{1 \leq i, j \leq n}$ denote the distance matrix between locations, and $F=(f_{ij})_{1 \leq i, j \leq n}$ the flow (or weight) matrix between facilities.  The objective is to find a bijection from facilities to locations that minimizes the overall cost, defined as the sum of distance–flow products:
\begin{equation} \label{eq:Obj_QAP}
\quad  Z_{\pi}(D, F) := \sum_{i=1}^{n} \sum_{j=1}^{n} d_{\pi(i)\pi(j)} f_{ij},
\end{equation} 
where $\pi \in S_n$ ranges over all permutations of $\{ 1, 2, \dots, n \}$. We denote the QAP  with an objective function \eqref{eq:Obj_QAP} as QAP($D,F$). The relationship between polynomially solvable special cases of QAPs and QUBOs is discussed in \cite[Section 3.3.1]{punnen2022quadratic}. Here, we highlight the equivalence between QAPs in these special cases and FCBQPs.

A matrix $A = (a_{ij})_{1 \leq i, j \leq n}$ is called a \emph{symmetric 0-1 product matrix} 
if there exists a vector $\alpha \in \{0,1\}^n$ such that $a_{ij} = \alpha_i \alpha_j \text{ for } 1 \leq i, j \leq n$.
Consider $D^{(\lambda)} = (d_{ij}^{(\lambda)})_{1 \leq i, j \leq n}$, a symmetric 0-1 product matrix defined by $d_{ij}^{(\lambda)} = x_i x_j$ for $1 \leq i, j \leq n$, where $x = (x_1 \ x_2  \dots  x_n)^\top  \in \{0,1\}^n$, with $x_i = 1$ for $1 \leq i \leq \lambda$ and $x_i = 0$ for the remaining indices. In this case:
$Z_\pi(D^{(\lambda)},F) = \sum_{i=1}^n \sum_{j=1}^nx_{\pi(i)}x_{\pi(j)}f_{ij}  = x_{\pi}^\top Fx_\pi$.
Hence, QAP($D^{(\lambda)}, F$) is equivalent to the following FCBQP:
  \begin{equation}  \label{eq:QAP_FCBQP}
    \min \left\{ x^\top F x \;\middle|\;  \sum_{i=1}^n x_i = \lambda, \ x \in \{0,1\}^n \right\}.
\end{equation}
The next theorem arises as a consequence of the above  QAP–FCBQP equivalence.
\begin{theorem} \label{QAP_FCBQP_theorem}
    The following cases of FCBQPs are polynomially solvable.
    \begin{itemize}
        \item[(i)] Let $q$ be a fixed natural number and let $P = (p_{ij})_{1 \leq i,j \leq q}$ be a matrix such that $p_{ii}+p_{jj} \leq 2p_{ij}$ for all $1\leq i, j \leq q.$ Then the FCBQP($Q,0,k$) can be solved in $\mathcal{O}(q^2q!+n \log n)$ time if $Q$ is a block matrix with block pattern $P$ (as in \cite{cela2015well}), i.e., (a) the set $\{1,\dots, n\}$ can be partitioned into $q$ (possibly empty) intervals $I_1, \dots, I_q$ such that for each $1 \leq t \leq q-1,$ all elements in $I_t$ are smaller than all elements in $I_{t+1},$ and (b) for all indices $1 \leq i,j \leq n,$ if  $i \in I_k$ and $j \in I_l,$ then $q_{ij}=p_{kl}.$
        \item[(ii)] \ The FCBQP($Q,0,k$) can be solved in $\mathcal{O}(n)$ time if $Q$ is a Toeplitz matrix 
        
        generated by a benevolent function (as defined in \cite{burkard1998quadratic}), that is, there exists a function
        $f:\{-n+1, \dots, n-1\} \to \mathbb{R}$ such that $q_{ij}=f(i-j),$ for all  $1 \leq i,j \leq n,$ and $f$ satisfies (a) $f(-i)=f(i)$ for all $1 \leq i \leq n-1,$ (b) $f(i) \leq f(i+1)$ for all $1 \leq i \leq \lfloor \frac{n}{2} \rfloor -1,$ and (c) $f(i) \leq f(n-i)$ for all $1 \leq i \leq \lceil \frac{n}{2} \rceil -1.$ 
    \end{itemize}
\end{theorem}
\begin{proof} As shown in \eqref{eq:QAP_FCBQP}, the QAP$(D^{(k)}, Q)$ is equivalent to the FCBQP($Q,0,k$), where $D$ is a symmetric 0-1 product. As a result, specific tractable cases of the QAP($D^{(k)}, Q$) also apply to the FCBQP($Q,0,k$).

Case $(i).$ If the pattern matrix $P$ satisfies $p_{ii}+ p_{jj} \leq 2 p_{ij}$ for $1 \leq i,j \leq q,$ then $P$ does not allow any \emph{bad ensemble} (see \cite[Corollary 4.2]{cela2015well}), and the QAP($D^{(k)}, Q$) can be solved in $\mathcal{O}(q^2q!+ n \log n)$ time (see \cite[Theorem 4.6]{cela2015well}). 

Case $(ii).$
    The symmetric 0-1 product matrix $D^{(k)}=(d_{ij}^{(k)})_{1 \leq i,j \leq n},$ where $d_{ij}^{(k)} =x_ix_j$ for $1 \leq i,j \leq n,$ and $x_i =1$ for $1 \leq i \leq k, x_i = 0$ otherwise, can be converted to a \textit{monotone Anti-Monge matrix} (see \cite [Definition 1.1]{burkard1998quadratic}) by the permutation $\pi \in S_n,$ defined by $\pi(i) := n+1-i, \text{ for } 1 \leq i \leq n$.
    After this permutation, we have $x_i = 1$ for $n-k +1 \leq i \leq n.$
    It is proven in \cite[Theorem 1.9]{burkard1998quadratic} that QAP($A,B$), when $A$ is a monotone Anti-Monge matrix, and $B$ is a  Toeplitz matrix generated by a benevolent function, can be solved by the permutation $\pi^* \in S_n,$ defined by: 
    \begin{equation*}
    \pi^*(i) := 2i-1 \text{ for } 1 \leq i \leq \lceil n/2 \rceil, \text{ and } \pi^{*}(n+1-i):=2i \text{ for } 1 \leq i \leq \lfloor n/2 \rfloor.
    \end{equation*}
    It is easy to see that $\pi^*$ produces the odd numbers in increasing order, followed by even numbers in decreasing order. It follows that 
    \begin{equation*}
     x_{\pi^*(i)} = 1 \text{ for } \Bigl\lfloor\frac{n-k +1}{2} \Bigr \rfloor < i < \Bigl \lceil \frac{n+k}{2}\Bigr \rceil + 1 \text{ (i.e., for $k$ consecutive variables)}.   
    \end{equation*} 
    Since the permutation $\pi$ does not affect the consecutiveness of variables, and the objective value remains the same for any $k$ consecutive variables, an optimal solution can be obtained by setting $k$ consecutive variables to one and the others to zero. 
    \end{proof}
\begin{remark}
Under the same conditions on the cost matrix as Theorem \ref{QAP_FCBQP_theorem}, the complexity of FCBQPs is reduced by a factor of $n$ compared to QUBOs  (see \cite[Theorem 3.13~$(i-ii)$]{punnen2022quadratic}).
\end{remark}
A matrix $Q=(q_{ij})_{1 \leq i,j \leq n}$ is said to be the \emph{distance matrix
 of a one-dimensional point set} if there exists $n$ real numbers $\alpha_1, \alpha_2, \dots, \alpha_n$ such that $q_{ij} = |\alpha_i - \alpha_j| \text{ for all } 1\leq i,j \leq n.$ 
This special case has attracted significant attention, notably in QAPs and dispersion problems. For such instances, QUBOs are shown to be solvable in pseudo-polynomial time via their equivalence to QAPs (see \cite[Theorem 3.13 ($iii$)]{punnen2022quadratic}). We notice that adding a fixed cardinality constraint makes this solvability more apparent.
A straightforward proof--independent of QAPs--is presented below for FCBQPs.
\begin{theorem}
If $Q$ is the distance matrix of a one-dimensional point set, then the FCBQP($Q,0,k$) can be solved in $\mathcal{O}(n \log n)$ time.
\end{theorem}
\begin{proof}
The problem FCBQP($Q,0,k$) is equivalent to finding a subset of $k$ points on a line that minimizes the total pairwise distance. Since the total pairwise distance increases as the selected points become more spread out, the optimal solution must consist of $k$ consecutive points. The problem can be solved through the following steps:
\begin{itemize}
    \item Sort the $n$ points $\alpha_1, \alpha_2, \dots, \alpha_n$ on the line in ascending order. Denote the sorted results as $\alpha_1^\ast, \alpha_2^\ast, \dots, \alpha_n^\ast$. This takes $\mathcal{O}(n \log n)$ time (e.g., using Merge Sort). 
    \item For $1 \leq i \leq n-k+1$, compute the total pairwise distance of $k$ consecutive points, starting from point $\alpha_i^\ast$, denoted by $D_i$. We have $D_i = \sum_{j=0}^{k-1}(2j -k+1) \alpha^\ast_{i+j}$. 
The initial value $D_1$ can be computed in $\mathcal{O}(k)$ time. Define $S_t \coloneqq \sum_{r=1}^t \alpha_r^\ast$ for $t = 1,\dots,n$, which can be computed in $\mathcal{O}(n)$ time.
Since
$D_{i+1} = D_i +(k-1)(\alpha^\ast_i + \alpha^\ast_{i+k}) - 2(S_{i+k-1} - S_i)$, each value $D_{i+1}$ can be updated from $D_i$ in constant time. Hence, all values $D_i$ can be computed  in $\mathcal{O}(n)$ time. The minimum value $D_i^\ast$ among these total pairwise distances corresponds to half the optimal value for the FCBQP($Q,0,k$), which can be found in $\mathcal{O}(n)$ time.
\end{itemize}
Therefore, the overall runtime is $\mathcal{O}(n \log n).$
\end{proof}
\section{Rank-based properties of the cost matrices} 
When $Q$ is a symmetric rank-one matrix, QUBOs are NP-hard in general (see \cite[Theorem 3.10]{punnen2022quadratic}), except when the vector $c$ in \eqref{eq:QUBO} is the zero vector (see \cite[Section 3.3.2.1]{punnen2022quadratic}).   The following theorem proposes our new result for this case in FCBQPs. 
\begin{theorem} 
 If $Q$ is a symmetric rank-one matrix whose entries all have the same sign, i.e., $Q = uu^\top$ or $Q=-uu^\top$ for some $u \in \mathbb{R}^n$ with $u_i \geq 0$ for all $i = 1, \dots, n$, then the FCBQP($Q,0,k$) can be solved in polynomial time.
 \end{theorem}
 \begin{proof}
 First, we consider the case $Q=uu^\top$ for some vectors $u \in \mathbb{R}^n$. The objective function of FCBQP$(Q,0,k)$ becomes:
$x^\top Qx = (u^\top x)^\top(u^\top x) = \bigl( \sum_{i=1}^nu_ix_i \bigl)^2$. 
Since the solution $x$ is $k$-sparse, minimization is achieved by setting $k$ variables with the smallest values of $u_i$ to one and the rest to zero. 

The case $Q = -u u^\top$ follows by similar arguments.
\end{proof}
\begin{remark}
Note that the optimal solutions of FCBQPs do not change when we add up a constant to all entries of $Q$. Therefore, we can always make all entries of $Q$ have the same sign (either all nonnegative or all nonpositive) by adding an appropriate constant.
\end{remark}
It is proven in \cite{allemand2001polynomial} that QUBOs with a maximization objective can be solved in polynomial time when the cost matrix is symmetric, positive semidefinite, and of fixed rank. Using the same strategy, we obtain a similar result for FCBQPs, as stated below.
\begin{theorem}
    If $Q$ is a symmetric, negative semidefinite matrix of a fixed rank $r$, i.e., there exists an $r \times n$ matrix $V$ such that $Q = -V^\top V,$ then the FCBQP($Q,0,k$)
    can be solved in polynomial time. 
\end{theorem}
\begin{proof}
    Solving the FCBQP($Q,0,k$) in this case is equivalent to solving:
    \begin{equation} \label{eq:fixed_rank}
    \begin{aligned}
     \max_{\sum_{i=1}^n x_i =k, \ x \in \{0,1\}^n}    f(x) = x^\top (-Q)x= x^\top V^\top V x = \sum_{i=1}^r \langle v_i, x \rangle^2
    \end{aligned}
    \end{equation}
    where $v_i^\top$ is the $i$-th row of the matrix $V,$ and $\langle a,b \rangle$ denotes the inner product of vectors $a$ and $b.$ Let $K := \{x \in [0,1]^n \ | \ \sum_{i=1}^nx_i = k\}$ be the relaxation of the feasible set in \eqref{eq:fixed_rank}.  Then $K$ is the intersection of the hypercube $[0,1]^n$ with the hyperplane $\sum_{i=1}^nx_i =k.$ Consider the linear map $x \mapsto z= Vx$ from $\mathbb{R}^n$ to $\mathbb{R}^r,$ and let $Z$ be the image of $K$ under this map. Since $K$ is convex, $Z$ is also a convex polytope. For every extreme point $\bar{z}$ of $Z,$ there exists an extreme point $\bar{x}$ of $K$ such that $\bar{z} = V \bar{x}.$ Denote by $S$ the set of extreme points of the polytope $Z.$ The optimal value $f^*$ of \eqref{eq:fixed_rank} then satisfies
    \begin{align*}
        f^* =  &\max \left\{ \sum_{i=1}^{r} \langle v_i, x \rangle^2  \middle|\ \sum_{i=1}^nx_i =k,  x \in \{0,1\}^n \right\} \\
        \stackrel{(1)}{=}&\max_{x \in K} \left\{ \sum_{i=1}^{r} \langle v_i, x \rangle^2 \right\} = \max_{z \in Z} \left\{ \sum_{i=1}^{r} z_i^2 \right\} = \max_{z \in S} \left\{ \sum_{i=1}^{r} z_i^2 \right\} ,
    \end{align*}
    where the equality $\stackrel{(1)}{=}$ follows from the convexity of the function $x \mapsto \sum_{i=1}^r \langle v_i,x\rangle^2 $. The problem is thus reduced to enumerating the extreme points of the polytope $Z$ in $\mathbb{R}^r.$ 

    The polytope $Z$ is a subpolytope of the \emph{zonotope}, which is the image of $[0,1]^n$ under the linear map $x \mapsto z=Vx$, as mentioned in \cite[Section 2]{allemand2001polynomial}. The extreme points of the polytope $Z$ form a subset of the extreme points of the zonotope.  The complexity of calculating the extreme points of the zonotope (see \cite[Section 2]{allemand2001polynomial}) yields that the worst-case time complexity for FCBQPs in this case is $\mathcal{O}(n^{r-1})$ for $r \geq 3,$ and  $\mathcal{O}(n^r)$ for $r \leq 2$.  
\end{proof}
 \section{Conclusion}
 We identified several structural conditions on the cost matrix where FCBQPs can be solved in polynomial time. In certain cases, we even showed that the fixed cardinality constraint improves tractability. Moreover, from a practical point of view, checking whether a matrix belongs to one of these categories (i.e., a sum matrix, block matrix, Toeplitz matrix, distance matrix of a one-dimensional point set, symmetric rank-one matrix, or symmetric negative semidefinite matrix of a fixed rank) can be done in polynomial time. Still, we note that there are also cases where the imposition of a cardinality constraint destroys tractability. For example, while QUBOs \eqref{eq:QUBO} are solvable in polynomial time when the associated Max-Cut graph is planar \cite{hadlock1975} or when the diagonal elements of the cost matrix are negative, and the rank is
fixed (see \cite[Theorem 3.21]{punnen2022quadratic}), this polynomial-time solvability is unlikely to be preserved once a fixed cardinality constraint is imposed.
 
\begin{acknowledgement}
The work is supported by the German Academic Exchange Service (DAAD).  We thank the
two anonymous referees for their constructive comments, which helped improve the
manuscript.
\end{acknowledgement} 

%
%
%

\end{document}